\documentclass[12pt]{amsart}

\usepackage{amssymb}
\usepackage{bm}
\usepackage{graphicx}
\usepackage{psfrag}
\usepackage[usenames,dvipsnames]{xcolor}
\usepackage{enumerate}
\usepackage{url}
\usepackage{mathptmx}
\usepackage{subcaption}
\usepackage{algpseudocode}

\usepackage{mathtools}

\usepackage{xy}
\input xy
\xyoption{all}

\newcommand{\excise}[1]{}

\newtheorem{thm}{Theorem}[section]
\newtheorem{lemma}[thm]{Lemma}

\newtheorem{cor}[thm]{Corollary}
\newtheorem{prop}[thm]{Proposition}
\newtheorem{conj}[thm]{Conjecture}

\newtheorem{prob}[thm]{Problem}

\theoremstyle{definition}
\newtheorem{alg}[thm]{Algorithm}
\newtheorem{example}[thm]{Example}
\newtheorem{remark}[thm]{Remark}
\newtheorem{defn}[thm]{Definition}

\newtheorem{notation}[thm]{Notation}

\numberwithin{equation}{section}

\newcommand{\ring}[1]{\ensuremath{\mathbb{#1}}}

\renewcommand\>{\rangle}
\newcommand\<{\langle}

\newcommand\QQ{\ring{Q}}
\newcommand\RR{\ring{R}}

\newcommand\ZZ{\ring{Z}}

\newcommand\cW{{\mathcal W}}

\DeclareMathOperator\Ap{Ap} 

\def\EE{{\mathbb E}}
\def\HH{{\mathbb H}}
\def\cF{{\mathcal F}}
\def\cN{{\mathcal N}}
\def\cS{{\mathcal S}}
\def\cE{{\mathcal E}}

\def\orbit{\operatorname{orbit}}
\def\codim{\operatorname{codim}}

\begin{document}

\mbox{}
\title{Wilf's conjecture in fixed multiplicity}

\author[Bruns]{Winfried Bruns}
\address{Institut f\"ur Mathematik\\Universit\"at Osnabr\"uck\\49074 Osnabr\"uck, Germany}
\email{wbruns@uos.de}

\author[Garc\'ia-S\'anchez]{Pedro Garc\'ia-S\'anchez}
\address{Departamento de \'Algebra\\Universidad de Granada\\18071 Granada, Espa\~na}
\email{pedro@ugr.es}

\author[O'Neill]{Christopher O'Neill}
\address{Mathematics Department\\San Diego State University\\San Diego, CA 92182}
\email{cdoneill@sdsu.edu}

\author[Wilburne]{Dane Wilburne}
\address{Department of Mathematics and Statistics\\York University\\Toronto, ON, Canada}
\email{drw@yorku.ca}

\date{\today}

\begin{abstract}
We give an algorithm to determine whether Wilf's conjecture holds for all numerical semigroups with a given multiplicity $m$, and use it to prove Wilf's conjecture holds whenever $m \le 18$.  Our algorithm utilizes techniques from polyhedral geometry, and includes a parallelizable algorithm for enumerating the faces of any polyhedral cone up to orbits of an automorphism group.  We also introduce a new method of verifying Wilf's conjecture via a combinatorially-flavored game played on the elements of a certain finite poset.  
\end{abstract}

\maketitle

\section{Introduction}
\label{sec:intro}

In what follows, let $\ZZ_{\ge 0}$ denote the set of non-negative integers.  

A \emph{numerical semigroup} $S$ is a subset of $\ZZ_{\ge 0}$ containing $0$ that is closed under addition and has finite complement in $\ZZ_{\ge 0}$ (this last condition is equivalent to requiring that the greatest common divisor of its elements is $1$).  Forty years ago~\cite{wilf}, Wilf conjectured that every numerical semigroup $S$ satisfies an inequality involving the following basic quantities:
\begin{itemize}
\item 
the \emph{conductor} of $S$, denoted $\operatorname c(s)$, which is the smallest integer $c$ such that $c + \ZZ_{\ge 0} \subset S$ (this is guaranteed to exist since $S$ has finite complement in $\ZZ_{\ge 0}$);

\item 
the number $\operatorname n(S)$ of elements of $S$ less than $\operatorname c(S)$ (called \emph{sporatic elements}); and

\item 
the \emph{embedding dimension} of $S$, denoted $\operatorname e(S)$, which equals the number of positive elements of $S$ that cannot be written as a sum of two positive elements of $S$ (called the \emph{atoms} or \emph{primitive elements} of $S$).  

\end{itemize}

One can think of Wilf's conjecture, stated below, as a bound on the ``density'' of the sporatic elements in terms of the number of primitive elements.  The problem has attracted the attention of many researchers, in part because of how easy it is to state.  Despite substantial progress, much of which has been made in this century, Wilf's conjecture remains open.  

\begin{conj}[Wilf]\label{conj:wilf}
For any numerical semigroup $S$, 
\[
\operatorname c(S) \le \operatorname e(S) \operatorname n(S).
\]
\end{conj}

The original aim of this project was to develop a computational method for verifying Wilf's conjecture for all numerical semigroups $S$ with fixed $\operatorname m(S) = \min(S \setminus \{0\})$, called the \emph{multiplicity} of~$S$ (note there are infinitely many numerical semigroups with each fixed multiplcity $\operatorname m(S) \ge 2$).  Our results, however, have more far-reaching consequences than merely verifying the conjecture in some new cases:\ we provide new tools from polyhedral geometry and enumerative combinatorics with which to approach the conjecture.  

Our main tool is the Kunz polyhedron $P_m$, introduced independently in~\cite{kunz} and~\cite{kunz-coordinates}, whose integer points are in one to one correspondence with numerical semigroups with multiplicity $\operatorname m(S) = m$ (the coordinates of these points are known as the Kunz coordinates of the semigroup).  The points in the interior of $P_m$ translate to a class of numerical semigroups for which Wilf's conjecture is known to hold~\cite{kunz-coordinates}, namely those with \emph{maximal embedding dimension}.  For each face $F$ of $P_m$, we reduce the task of checking Wilf's conjecture for all Kunz coordinates in the interior of $F$ to the problem of determining if a certain set of linear inequalities has integer solutions.  The primary computational hurdle in verifying Wilf's conjecture for a given multiplicity $m$ then becomes enumerating the faces of the Kunz polyhedron, which for $\operatorname m(S) \ge 13$ is a challenging computation.  

The primary contributions of this manuscript are as follows.  

\begin{itemize}
\item 
We present an algorithm for enumerating the faces of any polyhedral cone under the action of an automorphism group.  The details of our algorithm, including its implementation in the software package \texttt{Normaliz}~\cite{Normaliz}, can be found in Section~\ref{sec:facelattice}.  Although this work is motivated by the computation of the Kunz polyhedron, our algorithm, as well as its implementation in \texttt{Normaliz}, is not limited to this case.  

\item 
We prove that in the interior of any face of the Kunz polyhedron, the Kunz coordinates of all numerical semigroups violating Wilf's conjecture are determined by a system of linear inequalities (Section~\ref{sec:algorithm}).  Through a series of reductions, several of which result from special cases in which Wilf's conjecture is known to hold, we verify that Wilf's conjecture holds for all numerical semigroups $S$ with multiplicity $\operatorname m(S) \le 18$.  

\item 
We demonstrate that the unbounded faces of the Kunz polyhedron containing integer points are indexed by a family of finite partially ordered sets, called \emph{Ap\'ery posets} (Section~\ref{sec:kunz}), and introduce a combinatorial game played on the elements of a given Ap\'ery poset whose outcome yields a method of checking if all numerical semigroups with Kunz coordinates interior to the corresponding face satisfy Wilf's conjecture.  Section~\ref{sec:game} contains a description of the game, along with several examples of its use.  

\end{itemize}

Prior to this work, Wilf's conjecture was known to hold for $\operatorname m(S) \le 10$ by assembling several special cases.  At a talk in the summer of 2017, Eliahou~\cite{eliahou-imns} claimed to have a proof using graph theoretical methods that every numerical semigroup with $\operatorname m(S) \le 12$ satisfies Wilf's conjecture, though this work has yet to appear in the literature.

\section{Numerical semigroups and Wilf's conjecture}
\label{sec:wilf}

We say that a numerical semigroup is \emph{Wilf} if it satisfies Wilf's conjecture.  In this section, we introduce some additional concepts from the numerical semigroups literature, and survey some recent progress on Wilf's conjecture.  We direct the interested reader to~\cite{wilf-survey} for an exhaustive overview of the partial results obtained to date.  

Fix a numerical semigroup $S \subset \ZZ_{\ge 0}$.  A \emph{gap} of $S$ is a nonnegative integer outside of $S$, and the largest gap of $S$, denoted $\operatorname F(S)$, is known as its \emph{Frobenius number}.  In particular, we have $\operatorname c(S) = \operatorname F(S) + 1$.  We denote by $\operatorname G(S)$ the set of gaps of $S$; its cardinality $\operatorname g(S) = |\operatorname G(S)|$ is called the \emph{genus} of $S$.  A \emph{generating set} of a numerical semigroup $S$ is a subset $A \subset S$ with 
\[
S = \<A\> = \{a_1 + \dots + a_k : k \in \ZZ_{\ge 0}, a_1, \ldots, a_k \in A\},
\]
and every numerical semigroup $S$ admits a unique generating set $\mathcal A(S)$ that is minimal with respect to containment.  The elements of $\mathcal A(S)$ are the atoms of $S$ (and as such are sometimes also called the \emph{minimal generators} of $S$), and $\operatorname e(S) = |\mathcal A(S)|$.  

The \emph{Ap\'ery set} of an element $n \in S$ is the set 
\[
\Ap(S;n) = \{s \in S : s - n \notin S\}.
\]
It is well known that $\Ap(S;n)$ has precisely $n$ elements, each of which is distinct modulo $n$.  More precisely, $\Ap(S;n) = \{0, a_1, \ldots, a_{n-1}\}$, where $a_i = \min\{m \in S: m \equiv i \bmod n\}$. 

An integer $f$ is said to be a \emph{pseudo-Frobenius number} of $S$ if $f \notin S$ but $f + \mathcal A(S) \subset S$.  In particular, $\operatorname F(S)$ is a pseudo-Frobenius number of $S$. The cardinality of the set $\operatorname{PF}(S)$ of pseudo-Frobenius numbers of $S$ is the (Cohen-Macaulay) \emph{type} of $S$ and denoted $\operatorname t(S)$.  According to \cite[Theorem 20]{f-g-h}, for any numerical semigroup $S$, 
\begin{equation}\label{eq:wilf-type}
\operatorname c(S) \le (\operatorname t(S) + 1)\operatorname n(S).
\end{equation}
This implies that if $\operatorname e(S) > \operatorname t(S)$, then $S$ is Wilf.  This has the following consequences.
\begin{itemize}
\item 
If $\operatorname t(S) = 1$, then $S$ is Wilf.  Numerical semigroup with $\operatorname t(S) = 1$ are called \emph{symmetric}, and include all numerical semigroups with $\operatorname e(S) = 2$ (see \cite{f-g-h,numerical}).  

\item 
If $S$ is \emph{irreducible} (that is, if $S$ cannot be expressed as the intersection of two numerical semigroups properly containing it), then $S$ is Wilf.  Indeed, in this case, if $\operatorname F(S)$ is odd, then one can show $\operatorname t(S) = 1$, so $S$ is symmetric and thus Wilf.  On the other hand, if $S$ is irreducible and $\operatorname F(S)$ is even (we say $S$ is \emph{pseudo-symmetric} in this case), then one can show $\operatorname t(S) = 2$ and $\operatorname e(S) \ge 3$.  In either case, $S$ is Wilf by~\eqref{eq:wilf-type}.  

\item 
Any numerical semigroup with $\operatorname e(S) = 3$ has $\operatorname t(S) \le 2$, and thus is Wilf \cite[Chapter~1]{numerical}. 

\item 
If $\operatorname e(S) = \operatorname m(S)$, then $S$ is Wilf, since $\operatorname e(S) \le \operatorname m(S)$ and $\operatorname t(S) \le \operatorname m(S) - 1$ hold for every numerical semigroup (see, for instance, \cite[Chapter 1]{numerical}).  Such numerical semigroups are called \emph{maximal embedding dimension} due in part to the first inequality.	
\end{itemize}

Separately, Wilf's conjecture has been proved in numerous special cases~\cite{barucci,d-m,eliahou-mc,e-ma}, for instance, when $\operatorname c(S) \le 3\operatorname m(S)$, or $\operatorname c(S) \le 21$, or $\operatorname g(S) = (\operatorname F(S) + \operatorname t(S))/2$.  Of particular relevance to the results in this manuscript is~\cite{f-h}, wherein Fromentin and Hivert prove via computation that every numerical semigroup $S$ with $\operatorname g(S) \le 60$ is Wilf.  The key is that there are only finitely many numerical semigroups of a given genus.  The repository
\begin{center}
\url{https://github.com/hivert/NumericMonoid}
\end{center}
contains the number of numerical semigroups with each genus up to 70, though Wilf's conjecture has not been verified for the computed semigroups with genus above 60.

\section{The Kunz polyhedron and related polyhedra}
\label{sec:kunz}

In this section, we introduce the Kunz polyhedron, as well as two (new) closely related polyhedra, one of which is a pointed cone $C_m$.  The main results are Theorem~\ref{t:aperyposetfaces}, which gives a combinatorial interpretation of the faces of $C_m$, and Corollary~\ref{c:withinfaces}, which is a key ingredient to Algorithm~\ref{a:wilfmultiplicity} in verifying Wilf's conjecture for fixed multiplicity.  

Throughout the remainder of the paper, we will use basic terminology and facts from convex geometry, which we briefly summarize here. For more extensive treatments we refer the reader to Bruns and Gubeladze~\cite{BG} or Ziegler~\cite{Ziegler}.

An \emph{affine half-space} of $\RR^n$ is a subset $\{x \in \RR^n: \lambda(x) \ge \alpha\}$ for some linear form $\lambda\in(\RR^n)^*$ and some $\alpha \in\RR$.  The half-space is \emph{linear} if $\alpha = 0$, and \emph{rational} if $\alpha$ and the coefficients of $\lambda$ can be chosen in $\QQ$.  A \emph{polyhedron} is the intersection of finitely many affine half-spaces.  We denote by $P^\circ$ the topological interior of $P$.  A \emph{polytope} is a bounded polyhedron, whereas a (polyhedral) \emph{cone} is the intersection of finitely many linear half-spaces. A cone $C$ is \emph{pointed} if $x,-x\in C$ implies $x = 0$. A \emph{support(ing) hyperplane} of a polyhedron $P$ is a hyperplane $H$ such that $P$ is contained in one of the two closed half-spaces into which $\RR^n$ is decomposed by $H$.  A \emph{face} of $P$ is a subset $F = P\cap H$ where $H$ is a support hyperplane.  The polyhedron $P$ itself is considered an improper face.  The \emph{dimension} of $F$ is the dimension of its affine hull, and a face of dimension $k$ is called a $k$-face. A \emph{facet} is a face $F$ such that $\dim F=\dim P-1$, and a \emph{vertex} is a face of dimension $0$. Faces of polyhedra, polytopes and cones are themselves polyhedra, polytopes, and cones, respectively. The \emph{extreme rays} of a cone are its $1$-faces. It is important to note that every proper face of a polyhedron is the intersection of the facets in which it is contained; in particular, a polyhedron has only finitely many faces.

An affine half-space $\{x \in \RR^n : \lambda(x) \ge \alpha\}$ is \emph{rational} if the coefficients of $\lambda$ and $\alpha$ can be chosen in $\QQ$. A \emph{rational} polyhedron is the intersection of rational half-spaces. Faces of rational polyhedra are rational.

The \emph{$H$-representation} of a polyhedron $P$ is an expression of $P$ as an intersection of half-spaces.  A cone $C$ can equivalently be represented as $\{\alpha_1v_1+\dots+\alpha_mv_m: \alpha_1,\dots,\alpha_n\in\RR_+ \}$ for some $v_1,\dots,v_m\in\RR^n$, and by Minkowski's theorem, a polytope $P$ is the convex hull of its vertices; these expressions constitute the \emph{$V$-representations} of their respective polyhedra.  Any nonempty polyhedron $P$ equals the Minkowski sum of some polytope $Q$ and a cone $C$, i.e., $P = \{x + y : x \in Q, y \in C\}$.  The cone $C$ is unique, and called the \emph{recession cone} of $P$.

\begin{defn}\label{d:kunzpolyhedron}
Fix a numerical semigroup $S \subset \ZZ_{\ge 0}$, and let $m = \operatorname{m}(S)$ denote its multiplicity.  Writing $\Ap(S;m) = \{0, a_1,\ldots, a_{m-1}\}$ so that each $a_i = k_im + i$ for some positive integer $k_i$, we call $(k_1, \ldots, k_{m-1})$ the \emph{Kunz coordinates} of $S$.  It can be shown (see~\cite{kunz-coordinates}) that an integer vector $(x_1, \ldots, x_{m-1})$ are the Kunz coordinates of a numerical semigroup with multiplicity $m$ if and only if 
\begin{align*}
x_i           & \ge 1         & \text{for } & 1 \le i \le m-1,
\\
x_i + x_j     & \ge x_{i+j}   & \text{for } & 1 \le i \le j \le m-1 \text{ with } i + j < m, \text{ and}
\\
x_i + x_j + 1 & \ge x_{i+j} & \text{for } & 1 \le i \le j \le m-1 \text{ with } i + j > m,
\end{align*}
where the subscript of $x_{i+j}$ is interpreted modulo $m$.  In the remainder of the paper, whenever we write a variable $x_i$, we regard $i$ as a nonzero element of the cyclic group $\ZZ/(m)$.  

The polyhedron defined by the above inequalities is called the \emph{Kunz polyhedron} associated to $m$, which we will denote by $P_m$.  Numerical semigroups with multiplicity $m$ are in natural bijection with the integer points of $P_m$. Thus, we will sometimes identify a numerical semigroup with its Kunz coordinates, and by abuse of language, we will for instance say that $S$ lies in the interior of $P_m$ or is contained in a certain face of $P_m$.
\end{defn}

The following result is a well known fact that follows from Selmer's equalities, and implies that numerical semigroups with multiplicity $m$ and genus $g$ are in bijection with the integer points in the intersection of the polyhedron~$P_m$ and the hyperplane $x_1 + \cdots + x_{m-1} = g$.

\begin{lemma}\label{l:genusandfrob}
Fix a numerical semigroup $S$ with multiplicity $m$.  We have
\[
\operatorname{g}(S) = x_1 + \cdots + x_{m-1}  \text{ and }  \operatorname{F}(S) = \max\{mx_i + i - m : 1 \le i \le m - 1\},
\]
where $(x_1,\ldots,x_{m-1})$ denote the Kunz coordinates of $S$.  
\end{lemma}
\begin{proof}
By Selmer's equalities \cite{selmer} (or, more directly, by counting the number of gaps in each equivalence class modulo $m$), the genus of $S$ equals 
\[\frac{1}{m}\sum_{w\in \Ap(S;m)}w-\frac{m-1}2=\sum_{i=1}^{m-1}x_i,\]
and $\operatorname{F}(S)=\max\Ap(S;m)-m$ by \cite[Proposition 2.12]{numerical}.
\end{proof}

We are now ready to introduce the Kunz cone.  

\begin{defn}\label{d:kunzcone}
Fix an integer $m\ge 3$.  The \emph{relaxed Kunz polyhedron} is the set $P'_m$ of rational points $(x_1,\ldots,x_{m-1})\in\mathbb{R}^{m-1}$ satisfying
\begin{align*}
x_i + x_j     &\ge x_{i+j}   & 1 &\le i \le j \le m-1, \, i + j < m
\\
x_i + x_j + 1 &\ge x_{i+j} & 1 &\le i \le j \le m-1, \, i + j > m.
\end{align*}
The \emph{Kunz cone} is the set $C_m$ of points $(x_1,\ldots,x_{m-1})\in\mathbb{R}^{m-1}$ satisfying
\begin{align*}
x_i + x_j &\ge x_{i+j}   & 1 &\le i \le j \le m-1, \, i + j \ne m.
\end{align*}
\end{defn}

Proposition~\ref{p:relaxedvertex} presents several basic properties of $P_m'$ and $C_m$, and their precise relation to the Kunz polyhedron $P_m$.  In particular, as a consequence of Proposition~\ref{p:relaxedvertex}, $C_m$ is a pointed cone and a translation of $P_m'$.  

\begin{prop}\label{p:relaxedvertex}
For each $m\ge 3$ the following hold:
\begin{enumerate}[(a)]
\item 
the Kunz cone $C_m$ is contained in the positive orthant $\RR_+^{m-1}$;

\item 
the relaxed Kunz polyhedron $P_m'$ has unique vertex
\[
v = (-1/m,-2/m,\ldots,-(m-1)/m),
\]
and one has $x \in P_m'$ if and only if $x - v \in C_m$; and 

\item 
$C_m$ is the recession cone of $P_m$ and a translation of $P_m'$.
\end{enumerate}  
\end{prop}

\begin{proof}
Assume that $x_k < 0$ for some $k$.  Let $n$ be the order of $k$ in $\ZZ/(m)$.  Using the inequalities of the Kunz cone, one sees immediately that $x_j < 0$ for all $j = 1, \dots, n-1$.  If $n > 2$ we can write $1 \equiv 2 + (k-1) \bmod m$, and use the inequality $x_2 + x_{n-1}	\ge x_1$ to obtain the contradiction $x_1 < x_1$. If $n = 2$, or, equivalently, $k = m/2$, then we get the contradiction $x_1 < x_1$ since $x_{k+1} \le x_1 + x_m < x_1$ and $x_1 \le x_{k+1} + x_k$. This proves the first claim.

Next one readily checks that $v$ satsifies every defining inequality of $P_m'$ with equality. It follows immediately that $P_m'=\{v\}+C_m$ and that $v$ is the only vertex of $P_m'$.

If a nonempty polyhedron is defined by a system of inequalities, then the recession cone is defined by the associated homogeneous system. Since $C_m\subset \RR_+^{m-1}$ it is indeed the recession cone of $P_m'$ and $P_m$.
\end{proof}

\begin{remark}\label{r:relaxedvertex}
Proposition~\ref{p:relaxedvertex} also appeared in \cite[Proposition~1.1]{kunz}, though we have included the proof for the reader's convenience.  In the same manuscript, Kunz shows that the defining inequalities of the Kunz cone are irreduandant \cite[Proposition~1.2]{kunz}.  
\end{remark}

We now provide a characterization of the faces of the relaxed Kunz polyhedron that contain numerical semigroups.  

\begin{defn}\label{d:kunzposet}
A poset $P = (\ZZ/(m)\setminus\{0\}, \preceq)$ is an \emph{$m$-Kunz poset} if for distinct $i, j \in P$, we~have $i \preceq j$ implies $j - i \preceq j$.
\end{defn}

\begin{defn}\label{d:aperyposet}
Fix a numerical semigroup $S$ with multiplicity $m$, and write 
$$\Ap(S;m) = \{0, a_1, \ldots, a_{m-1}\}$$
so that $a_i \equiv i \bmod m$ for each $i$.  The \emph{Ap\'ery poset} of $S$ is a poset 
$$\mathcal P(S) = (\ZZ/(m)\setminus\{0\}, \preceq)$$ defined by $i \preceq j$ whenever $a_j - a_i \in S$.  Said another way, $\mathcal P(S)$ is the divisibility poset of~$S$ restricted to the nonzero elements of $\Ap(S;m)$ wherein each element is labeled with its equivalence class modulo $m$.  
\end{defn}

\begin{example}\label{e:mcnugget}
Let $S = \<6, 9, 20\>$.  The Hasse diagram of $\mathcal P(S)$ is depicted in Figure~\ref{f:mcnuggetapery}.  Each minimal element of $\mathcal P(S)$ represents one of the minimal generators of $S$ aside from~$6$.  Moreover, in this depiction, each cover relation corresponds to adding some minimal element of $\mathcal P(S)$ (indeed, each ``up-right'' edge corresponds to adding $2$, and each ``up-left'' edge corresponds to adding $3$).  
\end{example}

\begin{figure}[t!]
\begin{center}
\hfill
\hfill
\begin{subfigure}[t]{0.30\textwidth}
\begin{center}
\includegraphics[height=1.0in]{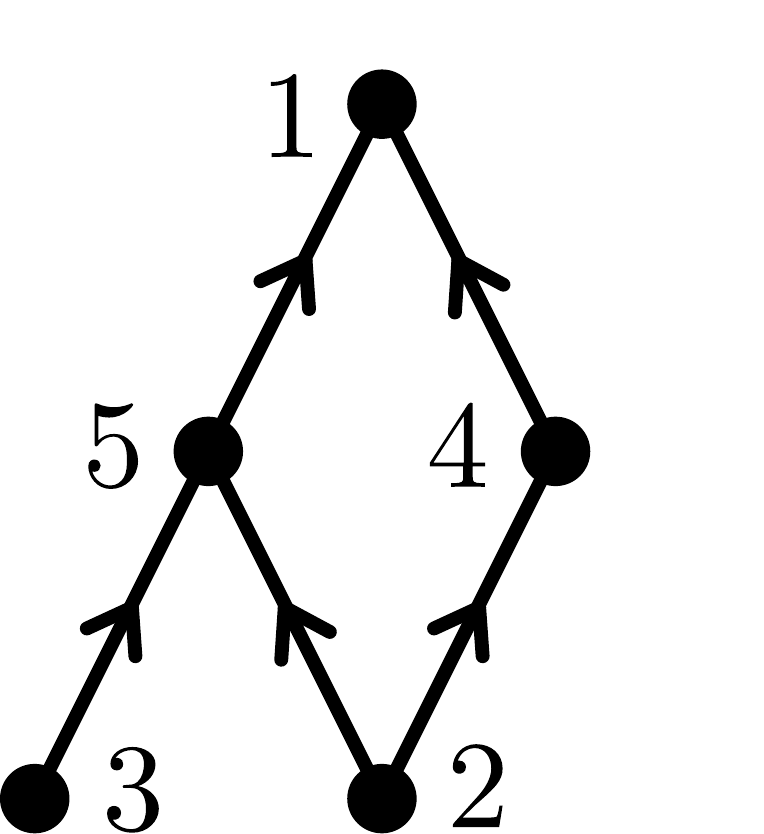}
\end{center}
\caption{}
\label{f:mcnuggetapery}
\end{subfigure}
\hfill
\begin{subfigure}[t]{0.20\textwidth}
\begin{center}
\includegraphics[height=1.0in]{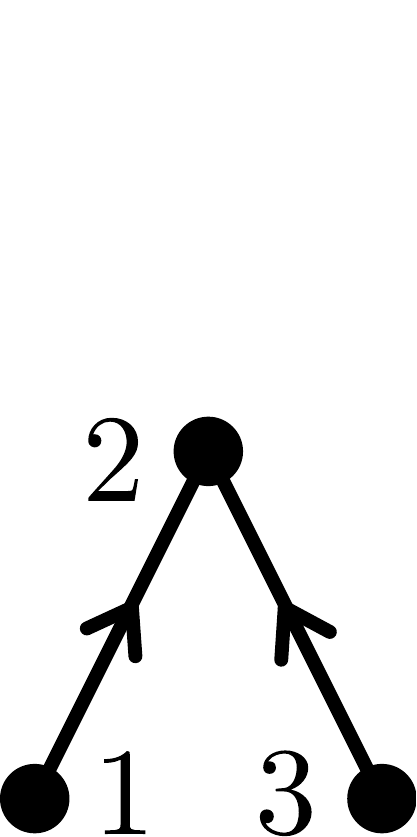}
\end{center}
\caption{}
\label{f:mkunzposetnonns}
\end{subfigure}
\hfill
\begin{subfigure}[t]{0.40\textwidth}
\begin{center}
\includegraphics[height=1.0in]{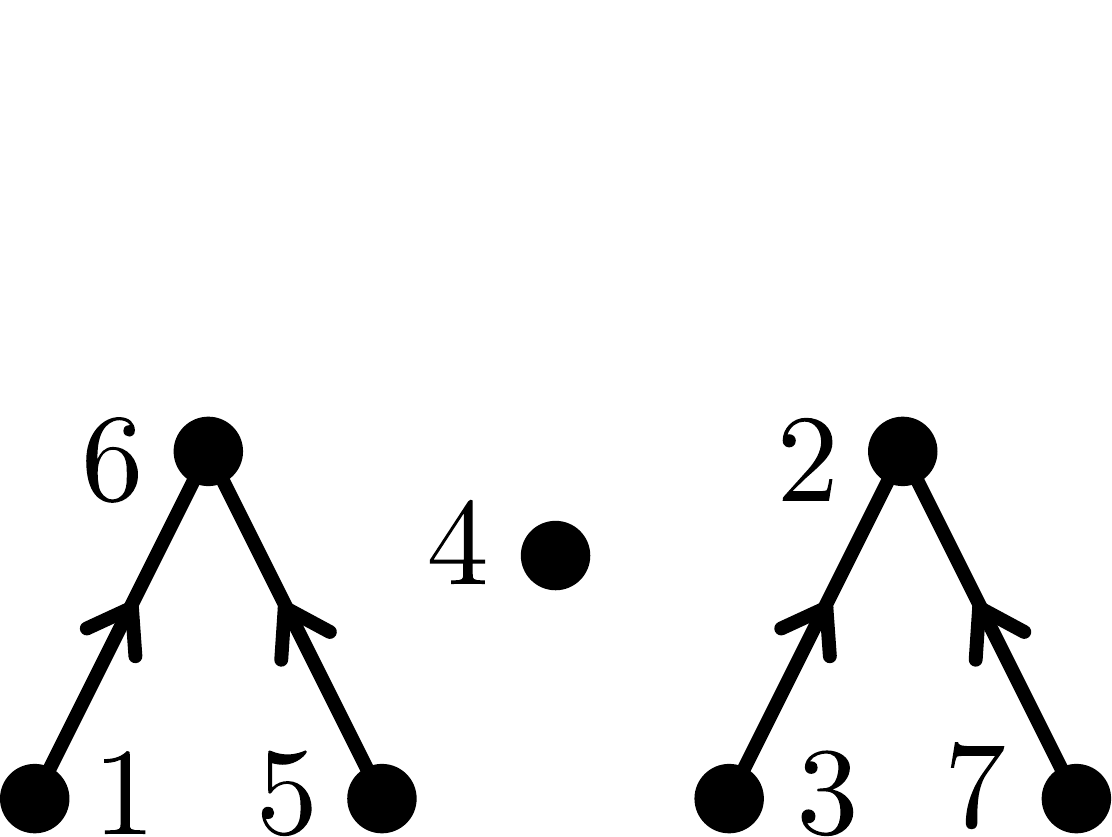}
\end{center}
\caption{}
\label{f:mkunzposetnonface}
\end{subfigure}
\end{center}
\hfill
\hfill
\caption{Sample $m$-Kunz posets from Examples~\ref{e:mcnugget} and~\ref{e:nofacebijection}.}
\end{figure}

\begin{lemma}\label{l:kunzposet}
The Ap\'ery poset of any multiplicity $m$ numerical semigroup $S$ is $m$-Kunz.  
\end{lemma}

\begin{proof}
Write $\Ap(S;m) = \{0, a_1, \ldots, a_{m-1}\}$ with $a_i \equiv i \bmod m$ for each $i$.  If $i \preceq j$ in $\mathcal P(S)$, then $a_j - a_i \in S$.  This means $a_j - a_i \in \Ap(S;m)$, as otherwise $a_j - a_i - m \in S$ and thus $a_j \notin \Ap(S;m)$.  Since $a_j - a_i \equiv j-i \bmod m$, we must have $a_j - a_i = a_{j-i}$, so we conclude $a_j - a_{j-i} = a_i \in S$ and thus $j-i \preceq  j$.
\end{proof} 

\begin{thm}\label{t:aperyposetfaces}
Two numerical semigroups $S$ and $T$ with multiplicity $m$ lie on the interior of the same face of $P_m'$ if and only if $\mathcal P(S) = \mathcal P(T)$.  
\end{thm}

\begin{proof}
Fix a face $F$ and a numerical semigroup $S \in F^\circ$.  Each defining inequality of $P_m$ holds with either equality or strict inequality for every point in $F^\circ$.  Moreover, $x_i + x_j = x_{i+j}$, with $i+j<m$, holds for the Kunz coordinates of $S$ if and only if $i \preceq i + j$ in $\mathcal P(S)$, and analogously $x_i + x_j = x_{i+j} - 1$, with $i+j>m$, holds for the Kunz coordinates of $S$ if and only if $i \preceq i + j$.  As such, the relations in $\mathcal P(S)$ are determined by the defining equations and strict inequalities of $F^\circ$.  
\end{proof}

The following result, which Kunz also observed in \cite[Section~2]{kunz}, relates the embedding dimensions of all numerical semigroups in the interior of the given face, and forms the crux of Algorithm~\ref{a:wilfmultiplicity}.  

\begin{cor}\label{c:withinfaces}
Fix a face $F \subset P_m'$.  
For any numerical semigroups $S, T \in F^\circ$, $\operatorname{e}(S) = \operatorname{e}(T)$ and $\operatorname{t}(S) = \operatorname{t}(T)$.  More specifically, $\operatorname{e}(S) - 1$ and $\operatorname{t}(S)$ count variables not appearing on the right and left hand sides of any defining equations of $F$, respectively.  
\end{cor}

\begin{proof}
Fix a face $F \subset P_m'$ and a numerical semigroup $S \in F^\circ$.  By the proof of Theorem~\ref{t:aperyposetfaces}, a~given defining inequality $x_i + x_j \ge x_{i + j}$, with $i+j<m$ holds with equality for $F$ if and only if $i \preceq i + j$ in $\mathcal P(S)$ (and similarly for $i + j > m$).  As such, $i$ is minimal (respectively maximal) in $\mathcal P(S)$ if and only if $x_i$ does not appear on the right (respectively left) hand side of any defining equalities of $F^\circ$.  

Now, \cite[Proposition 2.20]{numerical} implies the type of $S$ coincides with the cardinality of the set of maximal elements in $\mathcal{P}(S)$. It also follows easily that if $a$ is an atom of $S$ other than the multiplicity of $S$, then $a\in\Ap(S;m)$ and $a \bmod m$ is a minimal nonzero element of $\mathcal{P}(S)$.  This completes the proof.  
\end{proof}

\begin{notation}\label{n:withinfaces}
In view of Theorem~\ref{t:aperyposetfaces} and Corollary~\ref{c:withinfaces}, for each face $F \subset P_m'$ containing a numerical semigroup, we write $\operatorname e(F)$, $\operatorname t(F)$, and $\mathcal P(F)$ for embedding dimension, type, and Ap\'ery poset of any numerical semigroup in $F^\circ$.  
\end{notation}

\begin{example}\label{e:nofacebijection}
Some $m$-Kunz posets are not the Ap\'ery poset of any numerical semigroup.  For example, no numerical semigroup $S$ with multiplcity $m = 4$ can have an Ap\'ery poset with Hasse diagram in Figure~\ref{f:mkunzposetnonns}, as the relations imply its Ap\'ery set $\Ap(S;m) = \{0, a_1, a_2, a_3\}$ satisfies $2a_1 = a_2 = 2a_3$.  Despite this, $P_m'$ has a face corresponding to this $m$-Kunz poset under the correspondence in the proof of Theorem~\ref{t:aperyposetfaces} (a ray with recession cone $(1, 2, 1)$).  This ray simply does not contain any integer points.  

On the other hand, some faces of $P_m'$ contain no points with all positive entries (and thus also contain no numerical semigroups).  For instance, in the ray of $P_6'$ in the direction $(1, 0, 1, 0, 1)$, the second and fourth entries of every point are negative.  This face does not naturally correspond to an $m$-Kunz poset via the proof of Theorem~\ref{t:aperyposetfaces}.  

Lastly, some $m$-Kunz posets do not correspond to faces of $P_m'$.  For example, a face corresponding to the poset in Figure~\ref{f:mkunzposetnonface} would lie in exactly two facets of $P_8$, namely those with defining equations $x_1 + x_5 = x_6$ and $x_3 + x_7 = x_2$.  However, this implies 
\[
2x_6 + 2(x_2 + 1) = 2(x_1 + x_5) + 2(x_3 + x_7) = 2x_1 + 2x_3 + 2x_5 + 2x_7 \ge x_2 + x_6 + (x_2 + 1) + (x_6 + 1),
\]
so their intersection is also contained in the facets with defining equations $2x_1 = x_2$, $2x_3 = x_6$, $2x_5 = x_2$, and $2x_7 = x_6$.  
\end{example}

In view of Example~\ref{e:nofacebijection}, we pose the following.  

\begin{prob}\label{prob:facebijection}
Extend Theorem~\ref{t:aperyposetfaces} to characterize all faces of $P_m'$ (or,~equivalently, of~$C_m$).  
\end{prob}

\section{Verifying Wilf's conjecture for fixed multiplicity}
\label{sec:algorithm}

In this section, we provide an algorithm for determining whether Wilf's conjecture holds for every numerical semigroup with fixed multiplicity $m$.  The key to our algorithm is using Lemma~\ref{l:genusandfrob} and Corollary~\ref{c:withinfaces} to reduce Wilf's conjecture to checking for integer points in a finite list of rational polyhedra.  

Fix a face $F \subset P_m'$ and a numerical semigroup $S$ with Kunz coordinates $(x_1, \ldots, x_{m-1}) \in F^\circ$.  Using the fact that $\operatorname c(S) = \operatorname F(S) + 1$ and $\operatorname n(S) = \operatorname F(S) + 1 - \operatorname g(S)$, Wilf's conjecture can be reformulated as
\[
\operatorname F(S) + 1 \le \operatorname e(S)(\operatorname F(S) + 1 - \operatorname g(S)).
\]
Let $f \in [1, m-1]$ so that $mx_f + f$ is maximal, that is, 
\begin{equation}\label{eq:wilf1}
mx_i + i \le mx_f + f \text{ for every } i \ne f.
\end{equation}
By Corollary~\ref{c:withinfaces} every numerical semigroup in $F^\circ$ has identical embedding dimension $e$, so using Lemma~\ref{l:genusandfrob}, we may rewrite Wilf's inequality as
\begin{equation}\label{eq:wilf2}
mx_f + f - m + 1 \le e(mx_f + f - m - (x_1 + \cdots + x_{m-1}) + 1).
\end{equation}
We conclude that a face $F$ has a numerical semigroup in its interior that violates Wilf's conjecture if and only if, for some $f \le m - 1$ that is maximal in $\mathcal P(F)$, $F$ has an integer point satisfying inequalities~\eqref{eq:wilf1} and the negation of~\eqref{eq:wilf2}.  This yields Algorithm~\ref{a:wilfmultiplicity}.  
\begin{alg}\label{a:wilfmultiplicity}
Verify whether Wilf's conjecture holds for multiplicity~$m$.
\begin{algorithmic}
\Function{VerifyWilfsConjecture}{$m$}
\ForAll{$F$ face of $P_m'$}
	\State $R \gets$ defining equalities and strict inequalities of $F$
	\ForAll{$f = 1, \ldots, m-1$ with $f$ maximal in $\mathcal P(F)$}
		\State $R_f \gets$ inequalities~\eqref{eq:wilf1} and the negation of~\eqref{eq:wilf2}
		\If{region bounded by $R$ and $R_f$ contains a positive integer point}
			\State \Return False
		\EndIf
	\EndFor
\EndFor
\State \Return True
\EndFunction
\end{algorithmic}
\end{alg}

\begin{example}\label{e:wilfmultiplicity}
Let $F \subset P_6'$ denote the face containing $S = \<6, 9, 20\>$ from Example~\ref{e:mcnugget}, whose Ap\'ery poset is depicted in Figure~\ref{f:mcnuggetapery}.  Since $\operatorname e(F) = 3$, and $f = 1$ is the only maximal element, by Algorithm~\ref{a:wilfmultiplicity} we must check whether the system
\[
\begin{array}{c}
\begin{array}{c@{\qquad\quad}c@{\qquad\quad}c@{\quad}c}
\begin{array}{r@{}c@{}l}
x_1 - x_2 &{}\ge{}& 1 \\
x_1 - x_3 &{}\ge{}& 1 \\
x_1 - x_4 &{}\ge{}& 1 \\
x_1 - x_5 &{}\ge{}& 1 \\
\end{array}
&
\begin{array}{r@{}c@{}l}
2x_2 - x_4 &{}={}& 0 \\
x_2 + x_3 - x_5 &{}={}& 0 \\
-x_1 + x_2 + x_5 &{}={}& -1 \\
-x_1 + x_3 + x_4 &{}={}& -1 \\
\end{array}
&
\begin{array}{r@{}c@{}l}
-x_2 + 2x_4 &{}\ge{}& 0 \\
-x_4 + 2x_5 &{}\ge{}& 0 \\
-x_2 + x_3 + x_5 &{}\ge{}& 0 \\
-x_3 + x_4 + x_5 &{}\ge{}& 0 \\
\end{array}
&
\begin{array}{r@{}c@{}l}
2x_1 - x_2 &{}\ge{}& 1 \\
x_1 + x_2 - x_3 &{}\ge{}& 1 \\
x_1 + x_3 - x_4 &{}\ge{}& 1 \\
x_1 + x_4 - x_5 &{}\ge{}& 1 \\
\end{array}
\end{array}
\\\\[-0.6em]
-11x_1 + 3x_2 + 3x_3 + 3x_4 + 3x_5 \ge -6
\end{array}
\]
has any positive integer solutions.  More specifically, the inqualities in the first column come from~\eqref{eq:wilf1}, the equalities in the second column are the defining hyperplanes of $F$, the remaining two columns are the remaining inequalities in Definition~\ref{d:kunzpolyhedron}, and the final inequality is the negation of~\eqref{eq:wilf2}.  Some of the above inequalities have also been simplified using the fact that each $x_i$ is an integer.  Since $f = 1$ is the unique maximal element of $\mathcal P(F)$, the~infeasibility of this system implies $F$ contains no non-Wilf numerical semigroups.
\end{example}

\subsection*{Implementation}

The following refinements in Algorithm~\ref{a:wilfmultiplicity} result in significant reductions in runtime and memory. 

\begin{itemize}
\item 
By~\eqref{eq:wilf-type}, every numerical semigroup $S$ satisfies the inequality
\[
\operatorname{F}(S) + 1 \le (\operatorname{t}(S) + 1)(\operatorname{F}(S) - \operatorname{g}(S) + 1),
\]
so if $\operatorname{e}(S) > \operatorname{t}(S)$ then $S$ is Wilf.  As such, we can use Corollary~\ref{c:withinfaces} to eliminate certain faces of $P_m'$, namely those satisfying $\operatorname{e}(F) < \operatorname{t}(F) + 1$.

\item 
It is also known that if $S$ has ``high embedding dimension'', that is, if $2\operatorname e(S) \ge \operatorname m(S)$, then $S$ is Wilf, as proved by Sammartano \cite{samm}.  As such, any faces $F$ satisfying the $2\operatorname e(F) \ge m$ can be eliminated.  Additionally, for any two faces $F, F' \subset P_m'$ with $F \subset F'$, we have $\operatorname e(F) \le \operatorname e(F')$, so unlike the above item, once a face is encountered satisfying this inequality, every face containing it can be safely skipped.  

\item 
Prior to checking for integer points in a region $R$, we first check whether or not $R$ is feasibile (that is, whether $R$ is a nonempty subset of $\QQ^{m-1}$).  As it turns out, each region checked in Algorithm~\ref{a:wilfmultiplicity} for $m \le 18$ is infeasible. 

\end{itemize}
We refer to any face that cannot be skipped based on the considerations in the first two bullet points as a \emph{bad face}.

The above simplifications allowed Algorithm~\ref{a:wilfmultiplicity} to complete for $m \le 18$ using the software package \texttt{Normaliz} \cite{Normaliz} to (i) compute the face lattice of $P_m'$ (using the method developed in Section~\ref{sec:facelattice}) and (ii) subsequently enumerate the integer points in each region.  In addition to the reductions above, a custom build of \texttt{Normaliz} uses the automorphism group of the Kunz cone to simplify the face lattice computation, as well as automatically checking the feasibility of bad faces.  Our custom build can be downloaded from the following page.  
\begin{center}
\url{https://github.com/Normaliz/Normaliz/tree/wilf}
\end{center}
The repository contains input files for $11 \le m \le 19$.  See the file \texttt{ReadmeWilf} for more information on its usage.  

The output of our implementation of Algorithm~\ref{a:wilfmultiplicity} yields Theorem~\ref{t:m18}.  

\begin{thm}\label{t:m18}
For each $m \le 18$, every region tested in Algorithm~\ref{a:wilfmultiplicity} is empty.  In~particular, every numerical semigroup $S$ with $\operatorname m(S) \le 18$ is Wilf.  
\end{thm}

As noted above, each region tested in Algorithm~\ref{a:wilfmultiplicity} for $m \le 87$ is in fact empty.  With this in mind, we state the following conjecture, which implies Wilf's conjecture.  

\begin{conj}\label{conj:emptyregions}
For each $m \ge 3$, every region considered in Algorithm~\ref{a:wilfmultiplicity} is empty.  
\end{conj}

As a consequence of Theorem~\ref{t:m18}, we get the following result for numerical semigroups with high embedding dimension. 

\begin{cor}\label{c:delgado}
If $S$ is a numerical semigroup with $\operatorname m(S) - \operatorname e(S) \le 9$, then $S$ is Wilf.
\end{cor}
\begin{proof}
We are done by Theorem~\ref{t:m18} if $\operatorname m(S) \le 18$.  If $\operatorname m(S) \ge 19$, then $\operatorname{m}(S)-9\ge \operatorname{m}(S)/2$. Hence, by hypothesis, $\operatorname{e}(S)\ge\operatorname{m}(S)-9\ge \operatorname{m}(S)/2$, and thus $S$ is Wilf in light of \cite{samm}.
\end{proof}

If we copy the same argument used by Delgado in \cite[Remark 3.20]{wilf-survey}, together with the claims given in \cite{eliahou-imns}, then we can go a bit further, though we still do not have access to proofs of the results in \cite{eliahou-imns}.

\begin{cor}\label{c:delgado-2}
	If $S$ is a numerical semigroup with $\operatorname m(S) - \operatorname e(S) \le 12$, then $S$ is Wilf.
\end{cor}
\begin{proof}
If $\operatorname m(S) \le 18$, then we are done by Theorem~\ref{t:m18}.  If $\operatorname m(S) \ge 19$, then
\[
3\operatorname e(S) \ge 3(\operatorname m(S) - 12) = 3\operatorname m(S) - 36 \ge \operatorname m(S),
\]
and we are done by \cite{eliahou-imns}.  This completes the proof.  
\end{proof}

Note that both results, as mentioned in \cite{wilf-survey}, improve \cite[Thoerem 4.9]{md}.

\begin{remark}\label{r:wilfequality}
The following are known about numerical semigroups $S$ in which Wilf's inequality $\operatorname{c}(S) \le \operatorname{e}(S)\operatorname{n}(S)$ holds with equality:
\begin{enumerate}[(1)]
\item if $\operatorname{e}(S) = 2$, then equality holds; and
\item if $\operatorname{e}(S) = \operatorname{m}(S)$, then equality holds if and only if $x_1 = \cdots = x_{m-1}$.
\end{enumerate}
See \cite[Section 2.3]{wilf-survey}.  Using the computation of the face lattice discussed in Section~\ref{sec:facelattice}, we have verified that no additional cases of equality occur for $\operatorname{m}(S)\le 15$.
\end{remark}

\section{The computation of the face lattice}
\label{sec:facelattice}

In this section, we present an algorithm for computing the face lattice of a polyhedron (Algorithm~\ref{a:facelattice}) up to orbits under the action of a group of automorphisms, currently implemented in the version of \texttt{Normaliz}~\cite{Normaliz} that was dveloped for the application to the Wilf conjecture.  (We include a remark on the automorphism free version at the end of this section.) As many other computations, it is done on the cone over the polyhedron, and therefore it is enough to consider cones $C \subset \RR^d$ in the face lattice algorithm.  Replacing $\RR^d$ with the subspace $\RR C$, we can assume $C$ is full-dimensional in~$\RR^d$, and since the face lattice remains unchanged modulo the maximal linear subspace of $C$, we can further assume $C$ is pointed.  \texttt{Normaliz} handles these steps in preparatory coordinate transformations.  

\begin{remark}\label{r:automorphism}
If a finite group of automorphisms of $C$ is given, one can often speed up the algorithm significantly by computing only the orbits of the face lattice $\cF(C)$.  Moreover, storing only one face per orbit saves a considerable amount of memory.  If all faces are needed, the orbits can be easily expanded at the end.  In our computation of the face lattices of Kunz cones, this is a crucial observation.  

Suppose $C = C_m$ is the Kunz cone from Definition~\ref{d:kunzcone}.  The set of inequalities 
\[
x_i + x_{j-i} \ge x_j \qquad \text{for} \qquad i,j \in \ZZ/(m) \setminus \{0\}
\]
defining $C$ is stable under the multiplicative action of the group $(\ZZ/(m))^*$ of units modulo $m$ on the coordinates of the ambient space. Therefore, the cone and its set of integer points are stable under the action of $(\ZZ/(m))^*$, so this group action permutes the faces of $C$.  Fortunately, the subset of ``bad'' faces determined by Corollary~\ref{c:withinfaces} (see the discussion in Section~\ref{sec:algorithm}) is stable as well. Therefore it is enough to compute the orbits of the $(\ZZ/(m))^*$-action and select the ``bad'' orbits. However, the action of $(\ZZ/(m))^*$ does not carry over to the Kunz polyhedron, so the orbits must be expanded before testing for the existence of integer points in the critical area determined by Algorithm~\ref{a:wilfmultiplicity}.
\end{remark}

\begin{remark}\label{r:datastructure}
In view of the potentially large size of the computation, the choice of data structure is crucial. A facet $H$ of $C$ is given by a linear form defining the hyperplane that cuts $H$ out from $C$.  Faces $F$ have two natural representations:\ (i)~the set $\EE(F)$ of extreme rays passing through $F$, and (ii)~the set $\HH(F)$ of facets containing $F$.  Each uniquely defines $F$.  For faster computation, it is desirable to store both representations of $F$.  Since the number of facets of $F$ is moderate (at least for the Kunz cone), while the number of extreme rays of~$C$ reaches formidable values (see Table~\ref{tb:combinatorial}), we choose representation by $\HH(F)$, recomputing $\EE(F)$ whenever it is needed.  Storing $\EE(F)$ is forbidding --- already for $m = 16$, more than $1$~TB of RAM would be needed. Both representations are realized as bit vectors.  Fortunately, the computation of $\EE(F)$ for a face $F$ from a representation of $F$ as an intersection of facets takes relatively little time; see Remark~\ref{r:facelatticealg}(c).
\end{remark}

We say that a face $F$ of a cone $C$ is \emph{cosimplicial} if it is contained in exactly $c$ facets, where $c = \codim F = \dim C - \dim F$ is the codimension of $F$.  To motivate this terminology, consider the dual cone 
\[
C^* = \{\lambda \in (\RR^d)^* : \lambda(x) \ge 0 \text{ for all } x \in C\}
\]
of $C$ and dual face
\[
F_C^* = \{\lambda \in C^* : \lambda(x) = 0 \text{ for all } x \in F\}
\]
of $F$ (see \cite[Section~1B]{BG} for a discussion of duality).  One has $\codim F = \dim F_C^*$, and the linear forms defining the facets of~$C$ containing $F$ generate the extremal rays of $F_C^*$.  As such, $F$ is cosimplicial if and only if $F_C^*$ is simplicial.  It follows that any face $G$ containing a cosimplicial face $F$ is cosimplicial as well. One can say that cosimplicial faces are ``well-behaved'' in the face lattice computation, as is illuminated by the following proposition. 

As we assume that $C$ is full-dimensional, the linear forms defining the facets of $C$ are uniquely determined up to scaling by positive factors. As such, we can consider the facets as elements of the dual space $(\RR^d)^*$ so long as scaling factors can be neglected.

\begin{prop}\label{p:cosimp}
A face $F$ of $C$ is cosimplicial if and only if it is not the intersection of less than $c=\codim F$ facets.
\end{prop}

\begin{proof}	
By duality, the proposition is equivalent to the assertion that a cone $D$ of dimension $c$ is simplicial if and only if every set $S \subset \EE(D)$ of cardinality less than $c$ is contained in a proper face of $D$.  

If $D$ is simplicial, the latter is obviously true, and only the converse needs an argument. Suppose that $D$ is nonsimplicial.  We triangulate $D$ using only its extreme rays as rays in the triangulation.  Since $D$ is not simplicial, there are at least two dimension $c$ simplices in the triangulation whose intersection is a facet $G$ of both of them.  This means $D$ is its only face containing the $c - 1$ extreme rays of $G$.
\end{proof}

We now outline the contents of Algorithm~\ref{a:facelattice}.  The computation of the face lattice is based on some preparatory steps, starting from the definition of a cone $C$ by its facets.
\begin{enumerate}[(1)]
\item
Compute $\EE(C)$ by the existing vertex enumeration algorithm.

\item
For each facet $H$ of $C$, compute $\EE(H)$.

\item
Compute the set $\cS$ of cosimplicial extreme rays. 

\item
For each $x\in(\ZZ/(m))^*$, compute the permutation of $\HH(C)$ induced by $x$.

\end{enumerate}
The last step allows us to compute the orbit of an arbitrary face by applying the permutations to the facets in $\HH(F)$ (currently, this is only implemented in the specialized version of \texttt{Normaliz} for Kunz cones). The set $\cS$ allows us to recognize some cosimplicial faces~$F$ by checking whether $\EE(F) \cap \cS\neq \emptyset$.

The representation of a non-cosimplicial face as an intersection of facets is never unique (indeed, this follows immediately from Proposition~\ref{p:cosimp}).  For checking whether a given face (or orbit) has already been found, we use four ordered sets of faces:
\begin{enumerate}[(1)]
\item
the set $\cF$ of faces, each computed in the second to last round of the \textbf{while} loop or earlier;

\item
the set $\cW$ of faces found in the previous round and to be processed in the current round;

\item
the set $\cN$ of faces produced by the current round; and

\item
the set $\cE$ of intersections $G = F \cap H$ for a fixed face $F$ and $H \in \HH(C)$.

\end{enumerate}
A face $F$ in $\cF$, $\cW$ or $\cN$ is represented by $\HH(F)$, whereas each face $G$ in the short list~$\cE$ is represented by $\EE(G)$.  Lists of bit vectors are ordered lexicographically (based on fixed (but arbitrary) orderings of $\HH(C)$ and $\EE(C)$) so a list of size $n$ can be searched in $O(\log n)$ steps.  

\begin{alg}\label{a:facelattice}
Compute the orbits of the face lattice of a cone $C$.\label{algofacelattice}
\begin{algorithmic}	
\Function{FaceLattice}{$C$}
\State $\cF\gets\emptyset$, $\cW\gets \{C\}$, $c\gets 0$, $\cN\gets\emptyset$
\While{$\cW\neq\emptyset$}
\State $c\gets c+1$
\ForAll{$F\in\cW$}
\State $\EE(F)=\bigcap_{H\in\HH(F)}\EE(H)$
\State $\cE \gets \emptyset$
\ForAll{$H\in\HH(C)$ with $H\notin \HH(F)$}
\State $\cE	\gets \cE\cup \{ F\cap H\}$
\EndFor	
\ForAll{$G\in\operatorname{Max}_\subset\cE$} 			
\State compute $\HH(G)$
\State $\cN\gets \cN\cup\{\min \orbit(G)\}$
\EndFor
\EndFor
\State $\cF\gets \cF\cup \cW$, $\cW \gets \cN$, $\cN\gets \emptyset$
\EndWhile
\State \Return $\cF$
\EndFunction
\end{algorithmic}
\end{alg}

The set $\operatorname{Max}_\subset\cE$ is the set of the elements of $\cE$ that are maximal with respect to inclusion. Evidently these are the facets of $F$, and this proves the correctness of the algorithm:\ for every computed face $F$ of $C$, the facets of $F$ are also computed, and no proper subset $\cF$ of the full face lattice contains $C$ and is closed under taking facets of $F\in \cF$. Each orbit representative $G' \in \operatorname{orbit}(G)$ is chosen so that $\HH(G')$ is lexicographically minimal.

\begin{remark}\label{r:facelatticealg}
Before we refine Algorithm~\ref{a:facelattice}, we make several comments.  
\begin{enumerate}[(a)]
\item 
The number $c$, in addition to counting rounds of execution in the \textbf{while} loop, equals the codimension of the faces produced in the current round, as follows immediately by induction on $c$: if the face $F$ has codimension $c$, then the facets of $F$ each have codimension $c+1$. (This property will be somewhat relaxed below.)

\item 
The outer \textbf{for} loop is parallelized in \texttt{Normaliz} using \texttt{OpenMP} \cite{openmp}, a standard shared memory parallelization library. All threads must access the bit vectors $\EE(H)$ and the list $\cN$, but $\cN$ must be protected against simultaneous access since it is potentially changed by at least one thread.  Nevertheless, parallelization with $16$~threads is very reasonable; for $m = 14$, an efficiency of $\approx 45\%$ per thread is reached.  For comparison, with $8$~threads the efficiency is $\approx 70\%$, and with $4$ threads it is $\approx 77\%$. (Note that such measurements depend heavily on the workload of the machine).  

\item 
A profiler run for $m = 14$ shows that the computation of $\EE(F)$ for faces $F$ uses only $\approx6\%$ of the execution time for the face lattice.  The most time consuming task is lexicographic comparison of bit vectors at about $40\%$, followed by some system routines. The inclusion check takes $\approx 4\%$ .  

\item 
For a given $F$, it is extremely common to have $F \cap H = F \cap H'$ for facets $H \neq H'$ of $C$.  Therefore, it is useful to produce the set $\cE$ first; we obtain each intersection only once and can select the facets of $F$ by testing inclusion. 

\item 
Algorithm~\ref{a:facelattice} is designed for cones with small or at most moderate numbers of facets. Via dualization, it can also be effectively applied to cones with a small or moderate number of extreme rays (though this is not yet implemented in \texttt{Normaliz}).

\item
Although we speak of the ``face lattice'' throughout, Algorithm~\ref{a:facelattice} does not compute the lattice structure. Indeed, this would be impossible given the order of magnitude of the number of faces of the Kunz cone.  The standard version of \texttt{Normaliz} writes an output file in which each face $F$ is represented by a $0$-$1$-vector representing the facets that contain $F$, together with the codimension of $F$. The lattice structure can be easily derived from this data.

\end{enumerate}
\end{remark}

An obvious weakness of Algorithm \ref{a:facelattice} is that it ignores the commutativity of the intersection $F_1\cap F_2$.  Utilizing this fact should enable one to reduce the number of pairs $(F,H)$ that must be touched by the algorithm. To some extent this is achieved by using the following proposition.

\begin{prop}\label{p:facetrestriction}
For cosimplicial faces $F=H_1\cap\dots\cap H_c$, $H_1,\dots,H_c\in\HH(C)$, $H_1<\dots<H_c$,  it is sufficient to run the loop ``\textbf{for~all} $H \in \HH(C)$ with $H \notin \HH(F)$'' over those facets satisfying $H > H_c$, and to simultaneously replace $\operatorname{Max}_\subset\cE$ with the subset $\{G \in \cE : G \text{ is a facet of } F \}$.
\end{prop} 

\begin{proof}
Suppose $F$ is a face that has already been computed up to orbit. We must show that all facets $G$ of $F$ are computed up to orbit as well. Since all facets of~$C$ are evidently computed up to orbit, we can assume $\codim F\ge 1$. Set $E = H_1 \cap \cdots \cap H_{c-1}$ and $H = H_c$. 

We have $\codim G=\codim F+1=\codim E+2$. By the ``diamond property'' of the face lattice \cite[Theorem~2.7(c)]{Ziegler}, there exists exactly one more face $F'$ strictly between $E$ and $G$, which must be a facet of $E$ as well, meaning $F' = E \cap H'$ for some $H'\in \HH(C)$.  It follows that $F \cap H' = F' \cap H = F \cap F' = G$. The situation is depicted by Figure~\ref{f:diamond}.

If $H'>H$, then $G = F \cap H'$ is computed since $H'$ is not excluded by the condition in the proposition. In the case $H > H'$, however, we must be careful, since $F'$ has been computed only up to orbit. Let $\pi$ be the automorphism that maps $F'$ to the minimal face in its orbit. If $F'$ is not cosimplicial, then $\pi(F' \cap H) = \pi(F') \cap \pi(H)$ is computed since there is no restriction on the facets with which $\pi(F')$ is matched.  

If $F'$ is cosimplicial (necessarily of codimension $c$), then $\HH(\pi(F')) \le \HH(F')$ since we choose the lexicographically smallest face in the orbit. Moreover, $\HH(F') < \HH(F)$, since 
$$\max \HH(F') = \max(H_{c-1}, H') < H = \max \HH(F).$$
As such, we can assume $\pi(F') \cap \pi(H)$ is computed up to orbit by induction on the lexicographical order. It follows that $F \cap H$ is computed up to (the same) orbit, as desired.
\end{proof}

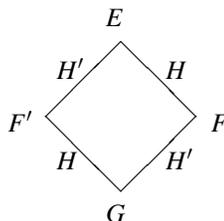
\begin{figure}[bt]
\unitlength1cm
\begin{footnotesize}
\begin{picture}(3,2.3)
\put(1,1){\line(1,1){1}}
\put(2,2){\line(1,-1){1}}
\put(1,1){\line(1,-1){1}}
\put(2,0){\line(1,1){1}}
\put(1.8,2.2){$E$}
\put(1.8,-0.4){$G$}
\put(0.5,0.8){$F'$}
\put(3.2,0.8){$F$}
\put(1.15,1.5){$H'$}
\put(2.6,1.5){$H$}
\put(1.15,0.3){$H$}
\put(2.6,0.3){$H'$}
\end{picture}
\end{footnotesize}
\caption{The diamond property in the proof of Proposition~\ref{p:facetrestriction}.}
\label{f:diamond}
\end{figure}

The inequality $\HH(F') < \HH(F)$ in the proof cannot be guaranteed (and rightfully is not used) if $F'$ is not cosimplicial. In the computation of orbits the main problem is to find invariants of the faces that behave equivariantly under the action of the automorphism group, or can at least be controlled with reasonable effort. In the computation of the Wilf cone, Proposition~\ref{p:facetrestriction} is already quite helpful because the number of cosimplicial faces (and their orbits) is rather high.  When $m = 14$, for instance, the Kunz cone has~$2,643,996$ cosimplicial orbits out of $3,506,961$.

The algorithm computes faces of codimension $c$ in round $c$ and not earlier (or later). There is, however, a catch in using Proposition \ref{p:facetrestriction}:\ one cannot select the facets of $F$ in $\operatorname{Max}_\subset\cE$ by checking inclusions unless all intersections $F \cap H$ with $H\in\HH(C)$ have been computed. In the Wilf version we proceed with $\operatorname{Max}_\subset\cE$ and allow the computation of a codimension $c$ face $F$ prior to round $c$. In order to avoid the re-computation of an orbit we look up $\cW\cup\cF$ before $F$ is added to $\cN$.

Table~\ref{tb:combinatorial} contains data on the Kunz cones and their face lattices. The increase in computation time with $m$ is not only due to the larger and larger face lattices, but also to the significant increase in the number of extreme rays.  We have computed their number for  $m = 19,20$ and $21$ to give a glimpse of the complexity one expects therein, let alone for higher values of $m$.  

\begin{table}
	\begin{tabular}{r|r|r|r|r|r|r|}
		$m$ &\# inequ&\# extr rays& \# orbits&\# bad orbits&\# faces&\# bad faces\\
		\hline
		\hline
		7&18&30  &71 &0&400&0\\ 
		\hline
		8&24&47&379&0&1,348&0\\
		\hline
		9&32&122& 1,104& 9 &6,508&54\\
		\hline
		10&40&225 &6,711&19& 26,682&74\\  
		\hline
		11& 50&   812&15,622&178&155,944&1,765\\
		\hline
		12&60& 1,864& 169,607 & 714& 669,794 & 2,791\\
		\hline
		13&72& 7,005& 365,881& 4,338& 4,389,234& 52,035\\
		\hline
		14&84&15,585& 3,506,961 & 15,251& 21,038,016 & 91,394\\
		\hline
		15&98&67,262& 17,217,534& 180,464& 137,672,474& 1,441,273\\
		\hline
		16&112&184,025& 94,059,396 & 399,380 & 751,497,188 & 3,184,022\\
		\hline
		17&128& 851,890&333,901,498 & 3,186,147 & 5,342,388,604 & 50,977,648\\
		\hline
		18& 144 & 2,158,379 & 4,712,588,473 & 17,345,725 & 28,275,375,292 & 104,071,319\\
		\hline
		19& 162 & 11,665,781 &?? & ?? & ?? & ??\\
		\hline
		20& 180 & 34,966,501  &?? & ?? & ?? & ??\\
		\hline	
		21& 200 & 169,543,084 &?? & ?? & ?? & ??\\
		\hline		
	\end{tabular}
	\vspace*{1.5ex}
	\caption{Combinatorial data of Kunz cones}
	\label{tb:combinatorial}
\end{table}

Table~\ref{tb:nmz_runtimes} gives execution times of the steps in the verification of Wilf's conjecture, and approximate values of the peak RAM usage.  The times listed for ``bad faces'' include the final transformations and the output times.  All runs were done with a parallelization of $32$~threads on a Dell R640 system with two Intel\texttrademark\ Xeon\texttrademark\ Gold 6152 (a total of $44$ cores) and 1 TB of RAM.  These times can vary quite a bit with the workload of the system. The~table indicates that both computation time and RAM usage are limiting factors in the computations.
\begin{table}
	\begin{tabular}{r|r|r|r|r|r|}
		$m$ & preparation& face lattice & bad faces & total time & $\approx$ RAM \\
		\hline
		\hline	
		11& --- & --- & --- & 0.7 s & 6 MB \\
		\hline
		12& --- & --- & --- & 2.0 s& 35 MB\\
		\hline
		13& 1 s& 2 s & 16 s& 19 s &  80 MB\\
		\hline
		14& 3 s & 20 s & 37 s & 1:0 m&  603 MB\\
		\hline
		15& 15 s & 3:335 m& 14 m& 17:59 m&  2.6 GB \\
		\hline
		16& 59 s& 54:39 m & 36 m & 1:30 h&  12 GB\\
		\hline
		17& 6:05 m& 19:35 h & 16:55 h& 36:36 h&  48 GB\\
		\hline
		18& 19:19 m& 27:13 d & 1:16 d    & 29:05 d &  720 GB\\
		\hline
	\end{tabular}
	\vspace*{1.5ex}
	\caption{\texttt{Normaliz} execution data in verifying Wilf's conjecture.}
	\label{tb:nmz_runtimes}
\end{table}

\begin{remark}\label{r:general_version}
An automorphism free version of the face lattice computation was released in \verb|Normaliz 3.7.0|.  Version \verb|3.8.0| will contain a substantially improved algorithm, which we forgo discussing at this point since (i)~it does not contribute to the Wilf computations, and (ii)~an algorithm by Kliem and Stump~\cite{K_and_S} posted to \verb|arXiv.org| after the first version of our paper appears to be faster than \verb|Normaliz 3.8.0|.  The Kliem and Stump algorithm differs greatly from ours, e.g.,\ by using a depth first search in the recursion and a faster implementation of bit vectors.  

\end{remark}

\section{Wilf's conjecture as a combinatorial game}
\label{sec:game}

The defining inequalities of any face in which Conjecture~\ref{conj:emptyregions} holds can be combined to yield Wilf's inequality.  We introduce a ``combinatorial game'' of sorts (Definition~\ref{d:wilfgame}), played with the facet description of each face (or, equivalently, the associated Kunz poset), the successful completion of which implies Wilf's conjecture holds for all numerical semigroups in that face (Theorem~\ref{t:wilfgame}).  

\begin{remark}\label{r:wilfgameadvantages}
Using the Wilf game in the pursuit of Wilf's conjecture has several potential advantages.  Some of the special classes of numerical semigroups for which Wilf's conjecture is known to hold have well-understood Ap\'ery posets (for instance, those in Corollary~\ref{c:wilfgame}).  The Wilf game provides a streamlined avenue for proving Wilf's conjecture in such cases, as the game only depends on the Ap\'ery posets of the underlying semigroup.  
Additionally, any counterexample to Wilf's conjecture must lie in a face for which the Wilf game is unwinnable, so one method of searching for such examples is by first locating such a poset.  Note that, using the machinery in Section~\ref{sec:algorithm}, once an unwinnable poset is located, \texttt{Normaliz} can be used to either locate an integer point whose corresponding numerical semigroup violates Wilf's conjecture, or verify computationally that none exist in the corresponding face.  
\end{remark}

\begin{defn}\label{d:wilfgame}
Fix an $m$-Kunz poset $P$ and a maximal element $f \in P$, and let $e - 1$ denote the number of minimal elements of $P$.  We define the \emph{Wilf game} of $P$, played on the set of all formal expressions $\sum_{i \in P} a_ix_i$ in variables $x_1, \ldots, x_{m-1}$ with $a_i \in \ZZ_{\ge 0}$.  A \emph{Wilf move} on a given expression is a replacement of the form 
$$x_i + x_{k-i} \to x_k$$
for some $i, k \in P$ with $i \prec k$.  The \emph{score} of a Wilf move $x_i + x_j \to x_k$ equals the sum~of:
\begin{enumerate}[(i)]
\item 
the net change in the number of summands ``to the right'' of $f$ (that is, variables whose index is greater than $f$), that is, the sum of
\begin{itemize}
\item 
$-1$ if $i > f$,

\item 
$-1$ if $j > f$, and

\item 
$+1$ if $k > f$;

\end{itemize}

\item 
$+1$ if $k < i$ (equivalently, if $k < j$); and

\item 
$+2$ if it is not one of the first $m - e$ moves performed.  

\end{enumerate}
A sequence of at least $m - e$ Wilf moves starting on the expression $ex_1 + \cdots + ex_{m-1}$ with initial score $m - 1 - f$ (that is, the number of distinct variables ``to the right'' of $f$) is~said to \emph{win the Wilf game} if the net score is non-negative.  
\end{defn}

\begin{example}\label{e:zero-wilfgame}
Let $S = \<3,5,7\>$, whose Ap\'ery poset $\mathcal{P}(S)$ has two elements, both of which are maximal. As such, there are no available Wilf moves.  However, $\operatorname m(S) - \operatorname e(S) = 0$, so the game is won with zero moves, as the initial score is either $3 - 1 - 1$ or $3 - 1 - 2$, both of which are nonnegative.

Notice that the same behavior occurs for every maximal embedding dimension numerical semigroup $S$, since in this case $\mathcal{P}(S)$ consists of $\operatorname{e}(S)-1$ incomparable elements (each both maximal and minimal). The Wilf game is won with zero moves for each maximal element, the net score being $m - 1 - f \ge 0$ for every $f\in (\ZZ/(m))^*$.
\end{example}

\begin{example}\label{e:wilfgame}
Let us return to the numerical semigroup $S = \<6, 9, 20\>$ from Example~\ref{e:mcnugget} (here, $m = 6$, $e = 3$, and $f = 1$).  The Wilf game on $P$ can be won with the sequence
$$\begin{array}{rcl@{}c@{}l@{}c@{}l@{}c@{}l@{}c@{}l}
3x_1 + 3x_2 + 3x_3 + 3x_4 + 3x_5
&\to& 4x_1 &{}+{}& 2x_2 &{}+{}& 3x_3 &{}+{}& 3x_4 &{}+{}& 2x_5 \\
&\to& 5x_1 &{}+{}& \phantom{1}x_2 &{}+{}& 3x_3 &{}+{}& 3x_4 &{}+{}& \phantom{1}x_5 \\
&\to& 6x_1 &{}+{}&  &{} {}& 3x_3 &{}+{}& 3x_4 &{} {}&  \\
&\to& 7x_1 &{}+{}&  &{} {}& 2x_3 &{}+{}& 2x_4 &{} {}&  \\
&\to& 8x_1 &{}+{}&  &{} {}& \phantom{1}x_3 &{}+{}& \phantom{1}x_4 &{} {}&  \\
&\to& 9x_1 &{} {}&  &{} {}&  &{} {}&  &{} {}&  \\
\end{array}$$
of 2 distinct Wilf moves ($x_2+x_5\to x_1$ and $x_3+x_4\to x_1$) each applied 3 times.  The net score is computed as follows:
\begin{enumerate}[(i)]
\item 
$-12$ points, since each Wilf move results in 2 less summands to the right of $f = 1$;

\item 
$+6$ points, since each move results in a variable with smaller index; and

\item 
$+4$ points for the extra $2$ moves beyond the initial $m - e = 4$.  

\end{enumerate}
The net score is thus $(m - 1 - f) - 12 + 6 + 4 = 2$.  Had each Wilf move been performed only twice, the net score would be $0$, still enough to win the Wilf game.  
Since $S$ is symmetric, $f$~is the unique maximal element of the Ap\'ery poset, and the strategy employed above (move everything directly to $x_f$) is precisely the one used in the proof of Corollary~\ref{c:wilfgame}(a).  
\end{example}

\begin{thm}\label{t:wilfgame}
Fix a face $F \subset P_m$ with corresponding Kunz poset $P$, and let $e = \operatorname{e}(F)$.  If, for each maximal element $f$, $P$ has a winning sequence of moves, then Conjecture~\ref{conj:emptyregions} (and thus Wilf's conjecture) holds for every numerical semigroup in $F^\circ$.  
\end{thm}

\begin{proof}
Fix a positive integer vector $(x_1, \ldots, x_{m-1}) \in F^\circ$ corresponding to the Kunz coordinates of some numerical semigroup $S$, and let $f \in P$ so that $mx_f + f = \operatorname{F}(S)$.  Rearranging the inquality~\eqref{eq:wilf2}, we must show 
$$ex_1 + \cdots + ex_{m-1} - m(e - 1)x_f + (e - 1)(m - 1 - f) \le 0.$$
Each Wilf move on $ex_1 + \cdots + ex_{m-1}$ corresponds to an equality of the form 
$$x_i + x_j = x_{i+j} \qquad \text{ or } \qquad x_i + x_j = x_{i+j-m} - 1$$
satisfied by all points in $F^\circ$.  Let $a_1x_1 + \cdots + a_{m-1}x_{m-1}$ denote the final expression resulting from some winning sequence of Wilf moves on $P$, let $s_3$ denote the number of Wilf moves used beyond the first $m - e$, let $s_2$ denote the number of Wilf moves of the latter form above, and let $s_1 = a_{f+1} + \cdots + a_{m-1}$, so that the net score is 
$$s_1 - (e-1)(m - 1 - f) + s_2 + 2s_3 \ge 0.$$
Using the fact that $x_i \le x_f$ for $i \le f$ and $x_i \le x_f - 1$ for $i > f$, we obtain
$$\begin{array}{r@{}c@{}l}
ex_1 + \cdots + ex_{m-1} - m(e - 1)x_f &{}={}&  -s_2 + a_1x_1 + \cdots + a_{m-1}x_{m-1} - m(e - 1)x_f \\
&{}\le{}& -s_1 - s_2 + (a_1 + \cdots + a_{m-1})x_f - m(e - 1)x_f \\
&{}={}& -s_1 - s_2 + (a_1 + \cdots + a_{m-1} - m(e-1))x_f \\
&{}={}& -s_1 - s_2 + (e(m - 1) - (m - e) - s_3 - m(e - 1))x_f \\
&{}={}& -s_1 - s_2 - s_3x_f \\
&{}\le{}& -s_1 - s_2 - 2s_3 \\
&{}\le{}& -(e - 1)(m - 1 - f),
\end{array}$$
from which the desired inequality immediately follows.  
\end{proof}

To demonstrate the utility of Theorem~\ref{t:wilfgame}, we rederive some known results.  

\begin{cor}\label{c:wilfgame}
The Wilf game is winnable on the Apery poset of a semigroup $S$ if
\begin{enumerate}[(a)]
\item 
$S$ is symmetric or

\item 
$S$ has maximal embedding dimension.  
\end{enumerate}
\end{cor}

\begin{proof}
Let $P$ denote the Ap\'ery poset of $S$, and let $m$, $e$, and $f$ be defined as before.  If~$S$ is symmetric, then $\operatorname{t}(S) = 1$, so $f$ is the unique maximal element of $P$.  As such, performing the move $x_i + x_{f - i} \to x_f$ for each $i \ne f$ yields a sequence of $m - 2$ moves.  Since $i < f$ if and only if $f - i < f$, each such move scores 0 (if $i < f$) or $-1$ (if $i > f$).  Combined with the starting score of $m - 1 - f$ and any additional points earned for extra moves, this yields a net non-negative score, thereby winning the Wilf game.  

The maximal embedding dimension case has been already treated in Example \ref{e:zero-wilfgame}.  
\end{proof}

\section{Acknowledgements}

The second author is supported by the project MTM2017--84890--P, which is funded by Ministerio de Econo\'ia y Competitividad and Fondo Europeo de Desarrollo Regional FEDER, and by the Junta de Andaluc\'ia Grant Number FQM--343.  The third author received support from the AMS-Simons travel grant to visit the second and fourth author's home institutions.  
The authors would also like to thank Eduardo Torres-Davila for his helpful comments.  


\end{document}